\documentclass[11 pt]{amsart}
\usepackage{amsmath}
\usepackage{amsfonts}
\usepackage{amssymb}
\usepackage{graphicx}
\usepackage{float, verbatim}
\usepackage{enumerate}
\usepackage{color}
\usepackage{hyperref}

\newcommand{\Assouad}{\dim_{\mathrm{A}}}

\renewcommand{\epsilon}{\varepsilon}

\newtheorem{thm}{Theorem}[section]

\newtheorem{lma}[thm]{Lemma}
\newtheorem{cor}[thm]{Corollary}
\newtheorem{defn}[thm]{Definition}

\newtheorem{ques}[thm]{Question}

\begin{document}

\title[Sets which avoid arithmetic progressions]{Dimensions of sets which uniformly avoid \\ arithmetic progressions}

\author[J. M. Fraser]{Jonathan M. Fraser}
\address{Jonathan M. Fraser\\
School of Mathematics \& Statistics\\University of St Andrews\\ St Andrews\\ KY16 9SS\\ UK  }
\curraddr{}
\email{jmf32@st-andrews.ac.uk}
\thanks{JMF is financially supported by a Leverhulme Research Fellowship (RF-2016-500).}

\author[K. Saito]{ Kota Saito }
\address{Kota Saito\\
Graduate School of Mathematics\\ Nagoya University\\ Furocho\\ Chikusa-ku\\ Nagoya\\ 464-8602\\ Japan }
\curraddr{}
\email{m17013b@math.nagoya-u.ac.jp}
\thanks{}

\author[H. Yu]{Han Yu}
\address{Han Yu\\
School of Mathematics \& Statistics\\University of St Andrews\\ St Andrews\\ KY16 9SS\\ UK  }
\curraddr{}
\email{hy25@st-andrews.ac.uk}
\thanks{HY is financially supported by the University of St Andrews.\\  \indent The authors thank Neal Bez for discussions which led to this collaboration.}



\date{}

\dedicatory{}
\begin{abstract}
We provide estimates for the dimensions of sets in $\mathbb{R}$ which uniformly avoid finite arithmetic progressions.  More precisely, we say $F$ uniformly avoids arithmetic progressions of length $k \geq 3$ if there is an $\epsilon>0$ such that one cannot find an arithmetic progression of length $k$ and gap length $\Delta>0$ inside the $\epsilon \Delta$ neighbourhood of $F$.  Our main result is an explicit upper bound for the Assouad (and thus Hausdorff) dimension of such sets in terms of $k$ and $\epsilon$.  In the other direction, we provide examples of sets which uniformly avoid arithmetic progressions of a given  length but still have relatively large Hausdorff dimension.

We also consider higher dimensional analogues of these problems, where arithmetic progressions are replaced with arithmetic patches lying in a hyperplane.  As a consequence we obtain a discretised version of a `reverse Kakeya problem': we show that if the dimension of a set in $\mathbb{R}^d$ is sufficiently large, then it closely approximates arithmetic progressions in every direction.
\end{abstract}

\maketitle

\section{Almost arithmetic progressions and dimension}

Arithmetic progressions are fundamental objects across mathematics and conditions which either force them to exist (or not exist) within a given set are of particular interest. For example, Szemer\'{e}di's seminal theorem \cite{szemeredi} states  that if $A\subset \mathbb{N}$ has positive upper density, then $A$ contains arbitrarily long arithmetic progressions. We say a set  $\{a_i\}_{i=0}^{k-1} \subset \mathbb{R}$ is an \emph{arithmetic progression (AP) of length $k$} if there exists $\Delta>0$ such that 
\[
a_i=a_0+i\Delta,
\]
for $i=1,2,\dots ,k-1$.  We say $\Delta$ is the \emph{gap length} of the arithmetic progression.    We are primarily interested in sets which \emph{uniformly avoid} arithmetic progressions and for this reason it is useful to introduce a weaker notion of `almost arithmetic progressions'.  In particular, given $\varepsilon \geq 0$ we say that $\{b_i\}_{i=0}^{k-1}\subset \mathbb{R}$ is a $(k,\varepsilon)$-AP if there exists an arithmetic progression $\{a_i\}_{i=0}^{k-1}$ with gap length $\Delta>0$ such that
\[
\lvert a_i - b_i \rvert \leq \varepsilon \Delta
\]
for all $i=0,2,\dots , k-1$.  Thus there is an arithmetic progression of length $k$ and gap length $\Delta>0$ inside the closed $\epsilon \Delta$ neighbourhood of a $(k,\varepsilon)$-AP.  We think of a set $F \subset \mathbb{R}$ as \emph{uniformly avoiding arithmetic progressions of length $k$} if, for some $\varepsilon>0$,  it does not contain any $(k,\varepsilon)$-APs.  Note that $(k,0)$-APs are simply the usual arithmetic progressions of length $k$.

The goal of this paper is to  quantify how `small' a set must be if it uniformly avoids arithmetic progressions.  We do this by providing explicit upper bounds for the dimensions of such sets.  There are numerous related notions of dimension appropriate for our purpose, but since the Assouad dimension is the biggest amongst the standard notions, estimating it from above will provide the strongest results.  We briefly recall the definition, but refer the reader to \cite{Fraser, Robinson} for more details.  

 For a non-empty bounded set $E \subset \mathbb{R}^d$ $(d \in \mathbb{N})$ and $r>0$, let $N (E, r)$ be the smallest number of open sets with diameter less than or equal to $r$ required to cover $E$.  The \emph{Assouad dimension} of a non-empty set $F\subseteq \mathbb{R}^d$ is defined by
\begin{eqnarray*}
\dim_\text{A} F & = &  \inf \Bigg\{ \  s \geq 0 \  : \ \text{      $ (\exists \, C>0)$ $(\forall \, R>0)$ $(\forall \, r \in (0,R) )$ $(\forall \, x \in F )$ } \\ 
&\,& \hspace{45mm} \text{ $N\big( B(x,R) \cap F , r\big) \ \leq \ C \bigg(\frac{R}{r}\bigg)^s$ } \Bigg\}
\end{eqnarray*}
where $B(x,R)$ denotes the closed ball centred at $x$ with radius $R$. It is well-known that the Assouad dimension is always an upper bound for the Hausdorff dimension, $\dim_\mathrm{H}$,  and (for bounded sets) the upper box dimension, $\overline{\dim}_\mathrm{B}$. We refer the reader to \cite{falconer} for more background on Hausdorff and box dimension.

Some connections between dimension and arithmetic progressions or almost arithmetic progressions are already known.  For example, \L aba and  Pramanik \cite{laba} showed that sets with Hausdorff dimension sufficiently close to 1 which support measures with certain Fourier decay necessarily contain arithmetic progressions of length 3 and Carnovale \cite{carnovale} extended this to longer arithmetic progressions. In the other direction, Shmerkin \cite{shmerkin} constructed examples of Salem sets $F \subset [0,1]$ of any dimension which do not contain arithmetic progressions of length 3.   Fraser and Yu \cite{FraserYu} proved that $F\subset\mathbb{R}$ has Assouad dimension $1$ if and only if  for all $k \geq 3$ and all $\varepsilon \in (0,1)$,  $F$ contains a $(k,\varepsilon)$-AP. As a corollary, they proved that if $F$ is a set of positive integers whose reciprocals form a divergent series, then for all $k \geq 3$ and all $\varepsilon \in (0,1)$,  $F$ contains a $(k,\varepsilon)$-AP.   The famous Erd\H{o}s-Tur\'{a}n conjecture on arithmetic progressions is that one can set $\epsilon=0$ here.

\section{Results for subsets of the line}

We can now state our main theorem on dimensions of sets which uniformly avoid arithmetic progressions, although we obtain a more general higher dimensional version of this result later, see Theorem \ref{higherdim}.  Here and throughout we write $\lceil x\rceil$ to mean the smallest integer greater than or equal to a real number $x \geq 0$.

\begin{thm}\label{Th1}
	Let $F\subset\mathbb{R}$ and fix an integer $k \geq 3$ and $\varepsilon \in (0,1)$.  If $F$ does not contain any $(k,\varepsilon)$-APs, then
	\[
	 \Assouad F\leq 1+\frac{\log (1-1/k)}{\log k \lceil 1/(2\varepsilon)\rceil}.
	\]
\end{thm}

We delay the proof of Theorem \ref{Th1} until Section \ref{proof1}.   Since the Assouad dimension is an upper bound for both the Hausdorff and box dimensions, this result also gives bounds on these dimensions.  Also, the converse of this theorem provides a useful check to prove the existence of approximations to certain APs.

Theorem \ref{Th1} asserts that if, for some $k \geq 3$ and $\varepsilon \in (0,1)$, a set $F \subset \mathbb{R}$ does not contain any $(k,\varepsilon)$-APs, then $\Assouad F < 1$.  This is not  \emph{not} true when $\varepsilon=0$ due to a result of Keleti \cite{keleti}. This result says that for every countable set $B\subset \mathbb{R}$ there exists a compact set with Hausdorff dimension 1 that intersects any similar copy of $B$ in at most two points. Therefore, for every $k\geq 3$ there exists a set with  full Hausdorff (and therefore Assouad) dimension that does not contain any $(k,0)$-APs.

The precise quantity we are interested in estimating in this paper is
\[
  \sup \{ \,  \dim F \, : \, \textup{$F \subset \mathbb{R}$ does not contain any $(k, \epsilon)$-APs} \}
\]
in terms of $\epsilon$ and $k$ where $\dim$ is either the Hausdorff or Assouad dimension.    Theorem  \ref{Th1} provides an upper bound and  our next theorem provides a lower bound.   See Section \ref{discussion} for more discussion on the sharpness of these bounds.

\begin{thm}\label{Th3}
	Fix an integer $k\geq 3$ and $\varepsilon \in (0,1)$ satisfying $\epsilon< (k-2)/4$. There exists a set $F \subset \mathbb{R}$ which does not contain any $(k,\varepsilon)$-APs  and
	\[
	\Assouad F = \dim_\textup{H}\:F =  \frac{\log 2}{\log \frac{2k-2-4\epsilon}{k-2-4 \epsilon}}.
	\]
\end{thm}

We delay the proof of Theorem \ref{Th3} until Section \ref{proof3}.

\subsection{Related notions of almost arithmetic progressions}

There are other possible ways to define and study `almost arithmetic progressions'. For example Lafont-McReynolds  \cite{lm} used the following notion: 

\begin{defn}[\cite{lm}]
	Fix an integer $k\geq 3$ and $\varepsilon >0$.  A set $\{a_i\}_{i=0}^{k-1} \subset \mathbb{R}$ is a $k$-term $\varepsilon$-almost arithmetic progression if
	\[
\left\lvert \frac{a_{i+1}-a_{i}}{a_{j+1}-a_{j}}-1 \right\rvert <\varepsilon
	\]
for $i,j=0, \dots, k-2$.
\end{defn}

In this section we simply remark that our main result also yields similar estimates for this related notion.
\begin{lma}
	Fix an integer $k \geq 3$ and $\varepsilon\in (0,1/2)$. Any $(k,\varepsilon)$-AP is a $k$-term $\epsilon'$-almost arithmetic progression in the sense of Lafont-McReynolds, for any $\epsilon'> \frac{4\varepsilon}{1-2\varepsilon}$.
\end{lma}

\begin{proof}
	Let $\{a_i\}_{i=0}^{k-1}$ be a $(k,\varepsilon)$-AP. Then for some $\Delta>0$ we have
	\[
	(1-2\varepsilon)\Delta \leq |a_{i+1}-a_i| \leq (1+2\varepsilon)\Delta,
	\]
for all $ i\in\{0,2,\dots,k-2\}$.  Therefore
	\[
	-\epsilon' < \frac{-4\varepsilon}{1+2\varepsilon}=\frac{1-2\varepsilon}{1+2\varepsilon}-1\leq\frac{a_{i+1}-a_{i}}{a_{j+1}-a_{j}}-1\leq \frac{1+2\varepsilon}{1-2\varepsilon}-1=\frac{4\varepsilon}{1-2\varepsilon} < \epsilon'
	\]
as required.
\end{proof}

One can combine this lemma with Theorem \ref{Th1} to obtain upper bounds for the dimensions of sets which do not contain any $k$-term $\varepsilon$-almost arithmetic progression  in the sense of Lafont-McReynolds since such sets must not contain any $(k,\epsilon')$-APs in our sense for any $0< \epsilon' < \varepsilon/(4+2\varepsilon)$.  We leave the details to the reader.

\section{Proof of Theorem \ref{Th1}} \label{proof1}

Fix an integer $k \geq 3$ and $\varepsilon \in (0,1)$.  Let $0<r<R$ and consider an arbitrary closed interval $I \subseteq \mathbb{R}$ of length $R>0$.  Initially, we assume that $1/(2\varepsilon)$ is an integer. First partition the interval $I$ into precisely $k/(2\varepsilon)$ many smaller intervals with length $2\varepsilon R/k$ and enumerate the smaller intervals from left to right with indices $i\in A=\{1,2, \dots,k/(2\varepsilon)\}$. Partition the indices into disjoint sets $A_j$ for $j= 0,1,2, \dots ,1/(2\varepsilon)-1$ defined by
	\[
	A_j=\{i\in A \mid  i\equiv j \  (\mathrm{mod} \  1/(2\varepsilon))\},
	\]
	and note that $|A_j|=k$ for each $j$.  	Write $I_i$ for the $i$-th interval of length $2\varepsilon R/k$, and suppose there is a $j$ such that
	\[
	I_i\cap F\neq\emptyset
	\]
for all $ i\in A_j$.  Choose $b_i\in I_{i}\cap F$ to form the set $\{b_i\}_{i\in A_j}$. Consider the arithmetic progression $\{a_i\}_{i\in A_j}$ where $a_i$ is the midpoint of the interval $I_i$, and observe that for all $i\in A_j$ we have
	\[
	|b_i-a_i|\leq\frac{\varepsilon R}{k}.
	\]
	Therefore $\{a_i\}_{i\in A_j}$ is a $(k,\varepsilon)$-AP with gap $R/k$ that is contained in $F$. Since we assumed that $F$ does \emph{not} contain any $(k,\varepsilon)$-APs, we conclude that for all $j$ at least one of the intervals $I_i$ ($i \in A_j$) must not intersect $F$.  Therefore, at most
\[
 \frac{k-1}{2 \varepsilon}
\]
intervals $I_i$ of length $2\epsilon R/k$ intersect $F$.  We now repeat the above argument within each interval $I_i$ of length $2\varepsilon R/k$ which \emph{does} intersect $F$.  We find that there are at most
\[
\left(\frac{k-1}{2\varepsilon}\right)^2
\]
intervals of length $(2\varepsilon/k)^2 R$ intersecting $I \cap F$.   We repeat this process $m$ times where $m$ is chosen such that $(2\varepsilon/k)^m R \approx r$.  More precisely, let $m=\left\lceil\frac{\log r/R}{\log 2\varepsilon/k}\right\rceil$ and note that 
\[
	\frac{(2\varepsilon)^m}{k^m}R\leq r,
	\]
	and it follows that
	\[
	N(I\cap F,r) \  \leq \  N\left(I\cap F,\frac{(2\varepsilon)^m}{k^m}R \right) \ \leq  \ \left(\frac{k-1}{2\varepsilon}\right)^m.
	\]
	For any $\delta>0$ we have
	\[
	m = \left\lceil\frac{\log r/R}{\log 2\varepsilon/k}\right\rceil\leq (1+\delta)\frac{\log R/r}{\log k/(2\varepsilon)}.
	\] 
provided $R/r$ is sufficiently large, which we may assume.  	Therefore, for any $x \in F$,
\[
	N(B(x,R)\cap F,r) \  \leq \ 2 \left(\frac{k-1}{2\varepsilon}\right)^m \  \leq  \ 2 \left(\frac{k-1}{2\varepsilon}\right)^{ (1+\delta)\frac{\log R/r}{\log k/(2\varepsilon)}}  \ = \ 2\left(\frac{R}{r}\right)^{ (1+\delta)\frac{\log (k-1)/(2\varepsilon)}{\log k/(2\varepsilon)}}.
	\]
It follows that

	\[
	\Assouad F \ \leq \  (1+\delta)\frac{\log (k-1)/(2\varepsilon)}{\log k/(2\varepsilon)}  = 1+\delta+\frac{\log(1-1/k)}{\log k/(2\varepsilon)},
	\]
	and as $\delta$ can be chosen arbitrarily close to $0$ we have
	\[
	\Assouad F\leq 1+\frac{\log(1-1/k)}{\log k/(2\varepsilon)}
	\]
as required.  If  $1/(2\varepsilon)$ is not an integer, then we can simply replace $\varepsilon$ by 
\[
\varepsilon' = \frac{1}{2\lceil 1/(2\varepsilon)\rceil}.
\]
 Observe that, since $\varepsilon' \leq \varepsilon$, if  $F$ does not contain any $(k,\varepsilon)$-APs, then it certainly does not contain any  $(k, \varepsilon')$-APs  the desired estimate follows.

\section{Proof of Theorem \ref{Th3}}  \label{proof3}

Before constructing the required examples, we prove a simple technical lemma.

\begin{lma}\label{Lem1}
Fix an integer $k\geq 3$ and $\varepsilon$ satisfying $0<\epsilon< (k-2)/4$.  Let $I\subset \mathbb{R}$ be a closed interval of length $|I|$, and let $J \subset I$ be an open interval of length $|J|$ satisfying
\[
 \frac{|I|(1+2 \epsilon)}{k-1-2\epsilon} < |J| <  |I|.
\]
 If $I \setminus J$ contains a $(k,\varepsilon)$-AP, then it must lie entirely to the left of $J$ or entirely to the right. 
\end{lma}

\begin{proof}
Suppose $I \setminus J$ contains a $(k,\varepsilon)$-AP with associated gap length $\Delta>0$.  It follows that
\[
(k-1-2\varepsilon) \Delta \leq |I|.
\]
Suppose also that this $(k,\varepsilon)$-AP intersects $I$ on both sides of $J$.  This means that one of the gaps much `bridge the hole' $J$ and so
\[
\Delta(1+2 \epsilon) \geq |J| >  \frac{|I|(1+2 \epsilon)}{k-1-2\epsilon}.
\]
Combining these estimates yields the desired contradiction.
\end{proof}

We are now able to construct sets satisfying the requirements of Theorem \ref{Th3}.   Fix an integer $k\geq 3$, a real number $\varepsilon$ satisfying $0<\epsilon< \min\{1,(k-2)/4\}$, and a sequence of increasing real numbers $c_m$ ($m \geq 1$) satisfying
\[
0<c_m  \nearrow   \frac{k-2-4 \epsilon}{2k-2-4\epsilon}.
\]
We construct $F$ via an iterative procedure.  Let $F_0 = [0,1]$ and for $m \geq 1$, let
\[
F_m = c_mF_{m-1} \cup  \left(c_m F_{m-1} +1-c_m \right).
\]
In particular, $F_m$ is a collection of $2^m$ closed intervals of length $c_1 c_2 \cdots c_m$.  Finally, let
\[
F = \bigcap_{m=0}^\infty F_m.
\]
We claim that $F$ cannot contain any $(k,\varepsilon)$-APs, due to Lemma \ref{Lem1}.  Indeed, suppose to the contrary and observe that each interval $I$ at stage $m \geq 0$ in the construction  splits up into two smaller intervals $I_1 \cup I_2$ at the next level where the `hole' $J= I \setminus (I_1 \cup I_2)$ has length
\[
|J| = |I|(1-2c_{m+1}) > \frac{|I|(1+2 \epsilon)}{k-1-2\epsilon}.
\]
Therefore, if a $(k,\varepsilon)$-AP is contained in $I$ then it is entirely contained inside either $I_1$ or $I_2$ by Lemma \ref{Lem1}.  By induction we conclude that any  $(k,\varepsilon)$-AP is a singleton, which is a contradiction.  Moreover, an elementary calculation which we omit yields
\[
\dim_\mathrm{A} F = \dim_\mathrm{H} F = \lim_{m \to \infty} \frac{\log 2}{-\log c_m} =  \frac{\log 2}{\log \frac{2k-2-4\epsilon}{k-2-4 \epsilon}}
\]
as required. Alternatively, $F$ can be viewed as a Moran construction and the given formula for the dimension is well-known.

\section{Higher dimensional analogues and discrete Kakeya problems}

In this section we consider an analogous problem in  higher dimensions. The proofs are similar to those presented for subsets of the line and so we omit most of the details. We consider subsets of $\mathbb{R}^d$ for an integer $d \geq 1$ and we replace `arithmetic progressions' with `arithmetic patches lying in particular subspaces'.  More precisely, let  $1 \leq m \leq d$ be an integer and let $ \mathbf{e} =  \{e_1, \dots, e_m\}$ be a set of orthogonal unit vectors.  We say $P \subseteq \mathbb{R}^d$ is an \emph{arithmetic patch with orientation $\mathbf{e}$, and of size $k$}, if there exists a `gap length' $\Delta >  0$ such that
\[
P = \left\{ t +   \Delta  \sum_{i=1}^m  x_i e_i  \ : \ x_i =0, \dots, k-1   \right\}
\]
for some $t \in \mathbb{R}^d$.  In particular, an arithmetic patch is a lattice consisting of $k^m$ points lying in a hyperplane parallel to the subspace spanned by $\mathbf{e}$.  Finally, for an integer $k \geq 2$, $\varepsilon \geq 0$, and an orientation $\mathbf{e}$,  we say $F \subseteq \mathbb{R}^d$ contains a $(k, \varepsilon, \mathbf{e})$-AP  if there exists an arithmetic patch $P$ with orientation $\mathbf{e}$,  size $k$, and gap length $\Delta >  0$ such that
\[
\sup_{x \in P} \inf_{y \in F} |x-y| \leq \epsilon \Delta.
\]

\begin{thm} \label{higherdim}
	Let   $m$ and $d$ be integers with $1 \leq m\leq d$,  $k \geq 2$ be an integer, and $\epsilon \in (0,1/\sqrt{d})$.  If  $F\subset\mathbb{R}^d$ is such that there exists an orientation $ \mathbf{e} =  \{e_1, \dots, e_m\}$  such that $F$ does not contain any  $(k, \varepsilon, \mathbf{e})$-APs, then
	\[
	\Assouad F\leq d+\frac{\log (1-1/k^m)}{\log k/\lceil \sqrt{d}/2\varepsilon\rceil}.
	\]
\end{thm}

\begin{proof}
	For simplicity of exposition, assume that $\mathbf{e}$ consists of the first $m$ elements in the standard basis for $\mathbb{R}^d$.  Assume $\sqrt{d}/(2\varepsilon)$ is an integer and let $0<r<R$.  Instead of an interval of length $R$, we consider a cube $Q$ of side length $R$ oriented with the coordinate axes.  We then decompose $Q$ into smaller cubes of side length $2\varepsilon R/(k \sqrt{d})$. There are $(k\sqrt{d}/(2\varepsilon))^d$ many of them and we  label them according to the lattice
	\[
	A^d=\left\{(z_1,z_2,\dots,z_d)\in\mathbb{Z}^d \ : \  1 \leq z_i  \leq k\sqrt{d}/(2\varepsilon) \right\}.
	\]
We now consider the `faces' parallel to the subspace spanned by $\mathbf{e}$.  In particular, we decompose the collection of  $(k\sqrt{d}/(2\varepsilon))^d$ smaller cubes into $(k\sqrt{d}/(2\varepsilon))^{d-m}$ faces each consisting of the $(k\sqrt{d}/(2\varepsilon))^m$ smaller cubes which share a particular common labeling in the final $(d-m)$ coordinates.

For each such face we   perform a  deleting procedure analogous to that used in the proof of Theorem \ref{Th1}.  Each face partitions into $(\sqrt{d}/(2\varepsilon))^m$ many `collections' which mimic $(k, \varepsilon, \mathbf{e})$-APs with `gap length' $R/k$.  Since the maximum distance from the centre of each cube to a point on the boundary is $\epsilon R/k$ and we assume $F$ does not contain any  $(k, \varepsilon, \mathbf{e})$-APs we can remove $(\sqrt{d}/(2\varepsilon))^m$ cubes from each of the faces (one from each `collection').  This means that at most
\[
\left(\frac{k\sqrt{d}}{2\varepsilon}\right)^d - \left(\frac{\sqrt{d}}{2\varepsilon}\right)^m  \cdot \left(\frac{k\sqrt{d}}{2\varepsilon} \right)^{d-m}  = \left(\frac{k\sqrt{d}}{2\varepsilon}\right)^d \left( 1 - \frac{1}{k^m}\right)
\]
of the smaller cubes can intersect $F \cap Q$.  Iterating this procedure within each cube which does intersect $F \cap Q$ as before yields the desired result.  
\end{proof}

We conclude by stating a simple corollary to Theorem \ref{higherdim}, which could be considered a discretised version of a `reverse Kakeya problem'.  The Kakeya problem is to prove that if a set $K \subseteq \mathbb{R}^d$ contains a unit line segment in every direction then it necessarily has Hausdorff dimension $d$.  Here we replace a unit line segment in a particular direction $e \in S^{d-1}$ by  a $(k, \varepsilon, \{e\})$-AP, i.e. an approximate arithmetic progression in direction $e$.  Our result then says that if a set has sufficiently large Assouad dimension, then it must contain an approximate arithmetic progression in every direction.  
\begin{cor} \label{kakeya}
	Let  $k \geq 2$ be an integer, and $\epsilon \in (0,1/\sqrt{d})$. If $F \subseteq \mathbb{R}^d$ and 
\[
	\Assouad F> d+\frac{\log (1-1/k)}{\log k/\lceil \sqrt{d}/2\varepsilon\rceil},
	\]
then $F$ contains a $(k, \varepsilon, \{e\})$-AP for every direction $e \in S^{d-1}$.
\end{cor}

\section{Future work and open questions} \label{discussion}

Theorem \ref{Th3} shows that Theorem \ref{Th1} is sharp in the sense that for a fixed $\epsilon \in (0,1)$, 
\[
\lim_{k\to \infty } \sup \{ \,  \dim_\mathrm{H} F \, : \, \text{$F \subset \mathbb{R}$ does not contain any $(k, \epsilon)$-APs} \} = 1.
\]
However, the following question is left as an interesting problem:

\begin{ques}
What is the value
\[
\lim_{\epsilon \to 0} \sup \{ \,  \dim F \, : \, \textup{$F \subset \mathbb{R}$ does not contain any $(k, \epsilon)$-APs} \}
\]
where $k \geq 3$ is fixed and $\dim$ is the Hausdorff dimension or the Assouad dimension?
\end{ques}
It follows from Theorem \ref{Th1} and Theorem \ref{Th3} that answer to the above question is bounded below by
\[
\frac{\log 2}{ \log \frac{2k-2}{k-2}}
\]
and above by 1.  It appears that this is related to the following problem in additive combinatorics.  Let $k \geq 3$ be an integer, and let $r_k(N)$ denote the largest cardinality of a set $A\subset \{1,\dots, N\}$ which does not contain any arithmetic progressions of length $k$. A very challenging problem is to estimate $r_k(N)$, and so far the best lower bound for general $k$ is given by O'Bryant \cite[Corollary 1]{bk} and for the best upper bound see Gowers \cite[Theorem 18.2]{gowers}.   In particular, the known bounds are good enough to conclude that $\log r_k( N)/\log N \to 1$ as $N \to \infty$.  For $\epsilon \in (0,1)$, let $r_k(\epsilon,N)$ denote the largest cardinality of a set $A\subset \{1,\dots, N\}$ which does not contain any $(k,\epsilon)$-AP. Clearly, we have $ r_k(\epsilon,N)\leq r_k(N)$ and the results of this paper imply that
 \[
\frac{\log 2}{\log \frac{2k-2-4\epsilon}{k-2-4 \epsilon}} \leq \limsup_{N\to\infty}\frac{\log r_k(\epsilon,N)}{\log N}\leq 1+\frac{\log (1-1/k)}{\log k \lceil 1/(2\varepsilon)\rceil}.
 \]
and it seems to be an interesting question to compute the precise value of this limit or to consider the more general problem of finding (sharp) bounds for $r_k(\epsilon,N)$.

Motivated by Corollary \ref{kakeya}, we also pose the following discrete analogue of the Kakeya problem:

\begin{ques}
Let $\epsilon \in (0,1)$ and suppose $F \subseteq \mathbb{R}^d$ contains a $(k, \varepsilon, \{e\})$-AP for every direction $e \in S^{d-1}$ and every $k \geq 3$.  Is it true that the Assouad dimension of $F$ is necessarily equal to $d$?
\end{ques}

A positive answer to this question would imply that every Kakeya set has Assouad dimension $d$ and Theorem \ref{higherdim} implies that the converse of this theorem is true, i.e. a set $F \subseteq \mathbb{R}^d$ with Assouad dimension $d$ necessarily contains a $(k, \varepsilon, \{e\})$-AP for every direction $e \in S^{d-1}$ and every $k \geq 3$.  Finally, we note that arithmetic progressions have been connected with the Kakeya problem before.  For example, Bourgain proved that if a set $F \subseteq \mathbb{R}^d$ contains a $(3, 0, \{e\})$-AP for every direction $e \in S^{d-1}$, then the box dimension of $F$ is at least $\frac{13}{25}(d-1)$, see \cite[Proposition 1.7]{bourgain}.

\begin{thebibliography}{5}

\bibitem[B]{bourgain}
J. Bourgain.
On the dimension of Kakeya sets and related maximal inequalities,
\emph{Geom. Funct. Anal.}, {\bf  9},  (1999), 256--282. 

\bibitem[C]{carnovale}
 M. Carnovale.
Long progressions in sets of fractional dimension,
{\em preprint}, (2013),   available at: http://arxiv.org/abs/1308.2919.


\bibitem[F]{falconer}
K.~J. Falconer.
{\em Fractal Geometry: Mathematical Foundations and Applications},
2nd Ed., John Wiley, Hoboken, NJ, (2003).

\bibitem[Fr]{Fraser}
J.~M. Fraser.
Assouad type dimensions and homogeneity of fractals,
{\em Trans. Amer. Math. Soc.}, {\bf 366}, (2014), 6687--6733.

\bibitem[FY]{FraserYu}
J.~M. Fraser and H. Yu.
Arithmetic patches, weak tangents, and dimension,
 {\em preprint}, (2016),   available at: http://arxiv.org/abs/1611.06960.

\bibitem[G]{gowers}
W. T. Gowers.
 A new proof of Szemer\'edi's theorem,
\emph{Geom. Funct. Anal.}, {\bf  11},  (2001),  465--588. 


\bibitem[K]{keleti} 
T. Keleti.
Construction of one-dimensional subsets of the reals not containing similar copies of given patterns,
{\em Anal. PDE},  {\bf 1}, (2008), 29--33. 

\bibitem[\L P]{laba} 
I. \L aba and M. Pramanik.
Arithmetic progressions in sets of fractional dimension,
{\em Geom. Funct. Anal.}, {\bf 19}, (2009), 429--456. 

\bibitem[LM]{lm} 
J.-F. Lafont and D.~B. McReynolds.
Primitive geodesic lengths and (almost) arithmetic progressions,
{\em preprint}, (2014),  available at: http://arxiv.org/abs/1401.7487.


\bibitem[O]{bk} 
K. O'Bryant.
Sets of integers that do not contain long arithmetic progressions,
{\em Electron. J. Combin.}, (2011), {\bf 18}, (2011).

\bibitem[R]{Robinson}
J. C. Robinson.
{\em Dimensions, Embeddings, and Attractors},
Cambridge University Press, (2011).

\bibitem[S]{shmerkin}
P. Shmerkin.
Salem sets with no arithmetic progressions,
{\em Int. Math. Res. Not. IMRN. (to appear)},  available at: http://arxiv.org/abs/1510.07596.


\bibitem[Sz]{szemeredi}
E. Szemer\'{e}di.
On sets of integers containing no $k$ elements in arithmetic progression,
Collection of articles in memory of Juri\v{i} Vladimirovi\v{c} Linnik. \emph{Acta Arith.}, {\bf 27}, (1975), 199--245.


\end{thebibliography}
\end{document}